\documentclass{amsart}
\usepackage{amsfonts}

\newtheorem{theorem}{Theorem}
\theoremstyle{plain}

\newtheorem{definition}{Definition}

\numberwithin{equation}{section}
\input{tcilatex}

\begin{document}
\title[On Hermite Hadamard-type inequalities]{On Hermite Hadamard-type
inequalities for strongly log-convex functions }
\author{Mehmet Zeki SARIKAYA}
\address{Department of Mathematics, \ Faculty of Science and Arts, D\"{u}zce
University, D\"{u}zce-TURKEY}
\email{sarikayamz@gmail.com}
\author{Hatice YALDIZ}
\email{yaldizhatice@gmail.com}

\begin{abstract}
In this paper, the notation of strongly log-convex functions with respect to 
$c>0$ is introduced and versions of Hermite Hadamard-type inequalities for
strongly logarithmic convex functions are established.
\end{abstract}

\maketitle

\section{Introduction}

The inequalities discovered by C. Hermite and J. Hadamard for convex
functions are very important in the literature (see, e.g.,\cite[p.137]%
{pecaric}, \cite{dragomir1}). These inequalities state that if $%
f:I\rightarrow \mathbb{R}$ is a convex function on the interval $I$ of real
numbers and $a,b\in I$ with $a<b$, then 
\begin{equation}
f\left( \frac{a+b}{2}\right) \leq \frac{1}{b-a}\int_{a}^{b}f(x)dx\leq \frac{%
f\left( a\right) +f\left( b\right) }{2}.  \label{E1}
\end{equation}%
The inequality (\ref{E1}) has evoked the interest of many mathematicians.
Especially in the last three decades numerous generalizations, variants and
extensions of this inequality have been obtained, to mention a few, see (%
\cite{alomari}-\cite{ngoc}) and the references cited therein.

\begin{definition}
The function $f:[a,b]\subset \mathbb{R}\rightarrow \mathbb{R}$, is said to
be convex if the following inequality holds%
\begin{equation*}
f(\lambda x+(1-\lambda )y)\leq \lambda f(x)+(1-\lambda )f(y)
\end{equation*}%
for all $x,y\in \lbrack a,b]$ and $\lambda \in \left[ 0,1\right] .$ We say
that $f$ is concave if $(-f)$ is convex.
\end{definition}

In \cite{pec2}, Pearce et. al. generalized this inequality to $r$-convex
positive function $f$ which defined on an interval $[a,b]$, for all $x,y\in
\lbrack a,b]$ and $t\in \lbrack 0,1]$%
\begin{equation*}
f\left( tx+(1-t)y\right) \leq \left\{ 
\begin{array}{ll}
\left( t\left[ f\left( x\right) \right] ^{r}+\left( 1-t\right) \left[
f\left( y\right) \right] ^{r}\right) ^{\frac{1}{r}}, & \text{if }r\neq 0 \\ 
\left[ f\left( x\right) \right] ^{t}\left[ f\left( y\right) \right] ^{1-t},
& \text{if }r=0.%
\end{array}%
\right.
\end{equation*}%
We have that $0$-convex functions are simply $\log $-convex functions and $1$%
-convex functions are ordinary convex functions.

Recently, the generalizations of the Hermite-Hadamard's inequality to the
integral power mean of a positive convex function on an interval $[a,b]$,
and to that of a positive $r$-convex function on an interval $[a,b]$ are
obtained by Pearce and Pecaric, and others (see \cite{pec2}-\cite{ngoc}).

A function $f:I\rightarrow \lbrack 0,\infty )$ is said to be log-convex or
multiplicatively convex if $\log t$ is convex, or, equivalently, if for all $%
x,y\in I$ and $t\in \left[ 0,1\right] $ one has the inequality:

\begin{equation}
f\left( tx+\left( 1-t\right) y\right) \leq \left[ f\left( x\right) \right]
^{t}\left[ f\left( y\right) \right] ^{1-t}.  \label{E2}
\end{equation}

We note that if $f$ and $g$ are convex and $g$ is increasing, then $g\circ f$
is convex; moreover, since $f=\exp \left( \log f\right) $, it follows that a
log-convex function is convex, but the converse may not necessarily be true 
\cite{pec2}. This follows directly from (\ref{E2}) because, by the
arithmetic-geometric mean inequality, we have

\begin{equation*}
\left[ f\left( x\right) \right] ^{t}\left[ f\left( y\right) \right]
^{1-t}\leq tf\left( x\right) +\left( 1-t\right) f\left( y\right)
\end{equation*}%
for all $x,y\in I$ and $t\in \left[ 0,1\right] $.

For some results related to this classical results, (see\cite{dragomir1},%
\cite{dragomir2},\cite{set1},\cite{set2}$)$ and the references therein.
Dragomir and Mond \cite{dragomir1} proved the following Hermite-Hadamard
type inequalities for the $\log $-convex functions:

\begin{eqnarray}
f\left( \frac{a+b}{2}\right) &\leq &\exp \left[ \frac{1}{b-a}%
\int\limits_{a}^{b}\ln \left[ f\left( x\right) \right] dx\right]  \label{z2}
\\
&\leq &\frac{1}{b-a}\int\limits_{a}^{b}G\left( f\left( x\right) ,f\left(
a+b-x\right) \right) dx  \notag \\
&\leq &\frac{1}{b-a}\int\limits_{a}^{b}f\left( x\right) dx  \notag \\
&\leq &L\left( f\left( a\right) ,f\left( b\right) \right)  \notag \\
&\leq &\frac{f\left( a\right) +f\left( b\right) }{2},  \notag
\end{eqnarray}%
where $G\left( p,q\right) =\sqrt{pq}$ is the geometric mean and $L\left(
p,q\right) =\frac{p-q}{\ln p-\ln q}$ $\left( p\neq q\right) $ is the
logarithmic mean of the positive real numbers $p,q$ $\left( \text{for }p=q,%
\text{ we put }L\left( p,q\right) =p\right) $.

Recall also that a function $f:I\rightarrow R$ is called strongly convex
with modulus $c>0,$ if%
\begin{equation*}
f\left( tx+\left( 1-t\right) y\right) \leq tf\left( x\right) +\left(
1-t\right) f\left( y\right) -ct(1-t)(x-y)^{2}
\end{equation*}%
for all $x,y\in I$ and $t\in (0,1).$ Strongly convex functions have been
introduced by Polyak in \cite{polyak}\ and they play an important role in
optimization theory and mathematical economics. Various properties and
applicatins of them can be found in the literature see (\cite{polyak}-\cite%
{angu}) and the references cited therein.

In this paper we introduce the notation of strongly logarithmic convex with
respect to $c>0$ and versions of Hermite-Hadamard-type inequalities for
strongly logarithmic convex with respect to $c>0$ are presented. This result
generalizes the Hermite-Hadamard-type inequalities obtained in \cite%
{dragomir1}\ for log-convex functions with $c=0.$

\section{Main Results}

We will say that a positive fuction $f:I\rightarrow \left( 0,\infty \right) $
is strongly log-convex with respect to $c>0$ if 
\begin{equation*}
f\left( \lambda x+\left( 1-\lambda \right) y\right) \leq \left[ f\left(
x\right) \right] ^{\lambda }\left[ f\left( y\right) \right] ^{1-\lambda
}-c\lambda \left( 1-\lambda \right) \left( x-y\right) ^{2}
\end{equation*}%
for all $x,y\in I$ and $\lambda \in (0,1).$ In particular, from the above
definition, by the arithmetic-geometric mean inequality, we have%
\begin{eqnarray}
f\left( \lambda x+\left( 1-\lambda \right) y\right)  &\leq &\left[ f\left(
x\right) \right] ^{\lambda }\left[ f\left( y\right) \right] ^{1-\lambda
}-c\lambda \left( 1-\lambda \right) \left( x-y\right) ^{2}  \label{E} \\
&\leq &\lambda f\left( x\right) +\left( 1-\lambda \right) f\left( y\right)
-c\lambda \left( 1-\lambda \right) \left( x-y\right) ^{2}  \notag \\
&\leq &\max \left\{ f\left( x\right) ,f\left( y\right) \right\} -c\lambda
\left( 1-\lambda \right) \left( x-y\right) ^{2}  \notag
\end{eqnarray}

\begin{theorem}
If a function $f:I\rightarrow \left( 0,\infty \right) $ be a strongly
log-convex with respect to $c>0$ and Lebesgue integrable on $I$, we have%
\begin{eqnarray}
f\left( \frac{a+b}{2}\right) +\frac{c\left( b-a\right) ^{2}}{12} &\leq &%
\frac{1}{b-a}\dint\limits_{a}^{b}G\left( f\left( x\right) ,f\left(
a+b-x\right) \right) dx  \label{E3} \\
&\leq &\frac{1}{b-a}\dint\limits_{a}^{b}f\left( x\right) dx  \notag \\
&\leq &L\left( f\left( a\right) ,f\left( b\right) \right) -\frac{c\left(
b-a\right) ^{2}}{6}  \notag \\
&\leq &\frac{f\left( a\right) +f\left( b\right) }{2}-\frac{c\left(
b-a\right) ^{2}}{6}  \notag
\end{eqnarray}%
for all $a,b\in I$ with $a<b.$
\end{theorem}

\begin{proof}
From (\ref{E}), we have%
\begin{eqnarray}
f\left( \lambda x+\left( 1-\lambda \right) y\right)  &\leq &\left[ f\left(
x\right) \right] ^{\lambda }\left[ f\left( y\right) \right] ^{1-\lambda
}-c\lambda \left( 1-\lambda \right) \left( x-y\right) ^{2}  \label{E4} \\
&\leq &\lambda f\left( x\right) +\left( 1-\lambda \right) f\left( y\right)
-c\lambda \left( 1-\lambda \right) \left( x-y\right) ^{2}.  \notag
\end{eqnarray}%
Since $f$ is a strongly log-convex function on $I$, we have for $x,y\in I$
with $\lambda =\frac{1}{2}$%
\begin{eqnarray}
f\left( \frac{x+y}{2}\right)  &\leq &\sqrt{f\left( x\right) f\left( y\right) 
}-\frac{c\left( x-y\right) ^{2}}{4}  \label{E5} \\
&\leq &\frac{f\left( x\right) +f\left( y\right) }{2}-\frac{c\left(
x-y\right) ^{2}}{4}  \notag
\end{eqnarray}%
i.e., with $x=ta+\left( 1-t\right) b$, $y=\left( 1-t\right) a+tb$,%
\begin{eqnarray}
&&f\left( \frac{a+b}{2}\right)   \label{E6} \\
&\leq &\sqrt{f\left( ta+\left( 1-t\right) b\right) f\left( \left( 1-t\right)
a+tb\right) }-\frac{c\left( b-a\right) ^{2}\left( 1-2t\right) ^{2}}{4} 
\notag \\
&\leq &f\left( ta+\left( 1-t\right) b\right) +f\left( \left( 1-t\right)
a+tb\right) -\frac{c\left( b-a\right) ^{2}\left( 1-2t\right) ^{2}}{4}. 
\notag
\end{eqnarray}%
Integrating the inequality (\ref{E6}) with respect to $t$ over $\left(
0,1\right) $, we obtain%
\begin{eqnarray*}
f\left( \frac{a+b}{2}\right)  &\leq &\frac{1}{b-a}\dint\limits_{a}^{b}\sqrt{%
f\left( x\right) f\left( a+b-x\right) }dx-\frac{c\left( b-a\right) ^{2}}{12}
\\
&\leq &\frac{1}{b-a}\dint\limits_{a}^{b}A\left( f\left( x\right) ,f\left(
a+b-x\right) \right) dx-\frac{c\left( b-a\right) ^{2}}{12},
\end{eqnarray*}%
and so for $\dint\limits_{a}^{b}f\left( x\right)
dx=\dint\limits_{a}^{b}f\left( a+b-x\right) dx,$%
\begin{eqnarray}
f\left( \frac{a+b}{2}\right) +\frac{c\left( b-a\right) ^{2}}{12} &\leq &%
\frac{1}{b-a}\dint\limits_{a}^{b}G\left( f\left( x\right) ,f\left(
a+b-x\right) \right) dx  \label{E7} \\
&\leq &\frac{1}{b-a}\dint\limits_{a}^{b}f\left( x\right) dx.  \notag
\end{eqnarray}%
Since $f$ is a strongly log-convex function on $I,$ for $x=a$ and $y=b,$ we
write 
\begin{eqnarray}
f\left( ta+\left( 1-t\right) b\right)  &\leq &\left[ f\left( a\right) \right]
^{t}\left[ f\left( b\right) \right] ^{1-t}-ct\left( 1-y\right) \left(
a-b\right) ^{2}  \label{z1} \\
&\leq &tf\left( a\right) +\left( 1-t\right) f\left( b\right) -ct\left(
1-t\right) \left( a-b\right) ^{2}.  \notag
\end{eqnarray}%
Integrating the inequality (\ref{z1}) with respect to $t$ over $\left(
0,1\right) $, we obtain,%
\begin{eqnarray*}
\frac{1}{b-a}\dint\limits_{a}^{b}f\left( x\right) dx &\leq &f\left( b\right)
\dint\limits_{0}^{1}\left[ \frac{f\left( a\right) }{f\left( b\right) }\right]
^{t}dt-c\left( b-a\right) ^{2}\dint\limits_{0}^{1}t\left( 1-t\right) dt \\
&\leq &f\left( a\right) \dint\limits_{0}^{1}tdt+f\left( b\right)
\dint\limits_{0}^{1}\left( 1-t\right) dt-c\left( b-a\right)
^{2}\dint\limits_{0}^{1}t\left( 1-t\right) dt,
\end{eqnarray*}%
and so%
\begin{equation}
\frac{1}{b-a}\dint\limits_{a}^{b}f\left( x\right) dx\leq L\left( f\left(
a\right) ,f\left( b\right) \right) -\frac{c\left( b-a\right) ^{2}}{6}\leq 
\frac{f\left( a\right) +f\left( b\right) }{2}-\frac{c\left( b-a\right) ^{2}}{%
6}.  \label{E8}
\end{equation}%
Thus, from (\ref{E7}) and (\ref{E8}), we obtain the inequality of (\ref{E3}%
). This completes to proof.
\end{proof}

\begin{theorem}
Let a function $f:I\rightarrow \lbrack 0,\infty )$ be a strongly log-convex
with respect to $c>0$ and Lebesgue integrable on $I$, then the following
inequality holds:%
\begin{eqnarray}
&&\frac{1}{b-a}\dint\limits_{a}^{b}f\left( x\right) f\left( a+b-x\right)
dx\leq f\left( a\right) f\left( b\right) +\frac{c^{2}\left( b-a\right) ^{4}}{%
30}  \notag \\
&&  \label{E9} \\
&&-\frac{4c\left( b-a\right) ^{2}}{\left[ \ln \left( f\left( b\right)
-f\left( a\right) \right) \right] ^{2}}\left[ A\left( f\left( a\right)
,f\left( b\right) \right) +L\left( f\left( a\right) ,f\left( b\right)
\right) \right]  \notag
\end{eqnarray}%
for all $a,b\in I$ with $a<b.$
\end{theorem}

\begin{proof}
Since $f$ is strongly log-convex with respect to $c>0$, we have that for all 
$t\in \left( 0,1\right) $%
\begin{eqnarray}
f\left( ta+\left( 1-t\right) b\right)  &\leq &\left[ f\left( a\right) \right]
^{t}\left[ f\left( b\right) \right] ^{1-t}-ct\left( 1-t\right) \left(
b-a\right) ^{2}  \notag \\
&&  \label{E10} \\
&\leq &tf\left( a\right) +\left( 1-t\right) f\left( b\right) -ct\left(
1-t\right) \left( b-a\right) ^{2}  \notag
\end{eqnarray}%
and 
\begin{eqnarray}
f\left( \left( 1-t\right) a+tb\right)  &\leq &\left[ f\left( a\right) \right]
^{1-t}\left[ f\left( b\right) \right] ^{t}-ct\left( 1-t\right) \left(
b-a\right) ^{2}  \notag \\
&&  \label{E11} \\
&\leq &\left( 1-t\right) f\left( a\right) +tf\left( b\right) -ct\left(
1-t\right) \left( b-a\right) ^{2}.  \notag
\end{eqnarray}%
Multiplying both sides of (\ref{E10}) by (\ref{E11}), it follows that%
\begin{eqnarray}
f\left( ta+\left( 1-t\right) b\right) f\left( \left( 1-t\right) a+tb\right) 
&\leq &f\left( a\right) f\left( b\right) +c^{2}\left( b-a\right)
^{4}t^{2}\left( 1-t\right) ^{2}  \notag \\
&&  \notag \\
&&-c\left( b-a\right) ^{2}t\left( 1-t\right) \left( f\left( b\right) \left[ 
\frac{f\left( a\right) }{f\left( b\right) }\right] ^{t}+f\left( a\right) %
\left[ \frac{f\left( b\right) }{f\left( a\right) }\right] ^{t}\right) .
\label{E12}
\end{eqnarray}%
Integrating the inequality (\ref{E12}) with respect to $t$ over $\left(
0,1\right) $, we obtain%
\begin{eqnarray}
&&\dint\limits_{a}^{b}f\left( ta+\left( 1-t\right) b\right) f\left( \left(
1-t\right) a+tb\right) dt\leq \dint\limits_{0}^{1}f\left( a\right) f\left(
b\right) dt+c^{2}\left( b-a\right) ^{4}\dint\limits_{0}^{1}t^{2}\left(
1-t\right) ^{2}dt  \notag \\
&&  \label{E13} \\
&&-c\left( b-a\right) ^{2}f\left( b\right) \dint\limits_{0}^{1}t\left(
1-t\right) \left[ \frac{f\left( a\right) }{f\left( b\right) }\right]
^{t}dt-c\left( b-a\right) ^{2}f\left( a\right) \dint\limits_{0}^{1}t\left(
1-t\right) \left[ \frac{f\left( b\right) }{f\left( a\right) }\right] ^{t}dt 
\notag \\
&&  \notag \\
&=&\dint\limits_{0}^{1}f\left( a\right) f\left( b\right) dt+c^{2}\left(
b-a\right) ^{4}\dint\limits_{0}^{1}t^{2}\left( 1-t\right) ^{2}dt-c\left(
b-a\right) ^{2}f\left( b\right) I_{1}-c\left( b-a\right) ^{2}f\left(
a\right) I_{2}.  \notag
\end{eqnarray}%
Integrating by parts for  $I_{1}$ and $I_{2}$ integrals, we obtain%
\begin{eqnarray}
&&I_{1}=\dint\limits_{0}^{1}t\left( 1-t\right) \left[ \frac{f\left( a\right) 
}{f\left( b\right) }\right] ^{t}dt  \notag \\
&&  \notag \\
&=&\left. t\left( 1-t\right) \frac{1}{\ln \left[ \frac{f\left( a\right) }{%
f\left( b\right) }\right] }\left[ \frac{f\left( a\right) }{f\left( b\right) }%
\right] ^{t}\right\vert _{0}^{1}-\frac{1}{\ln \left[ \frac{f\left( a\right) 
}{f\left( b\right) }\right] }\dint\limits_{0}^{1}\left( 1-2t\right) \left[ 
\frac{f\left( a\right) }{f\left( b\right) }\right] ^{t}dt  \notag \\
&&  \label{E14} \\
&=&-\frac{1}{\ln \left[ \frac{f\left( a\right) }{f\left( b\right) }\right] }%
\left[ \left. \left( 1-2t\right) \frac{1}{\ln \left[ \frac{f\left( a\right) 
}{f\left( b\right) }\right] }\left[ \frac{f\left( a\right) }{f\left(
b\right) }\right] ^{t}\right\vert _{0}^{1}+\frac{2}{\ln \left[ \frac{f\left(
a\right) }{f\left( b\right) }\right] }\dint\limits_{0}^{1}\left[ \frac{%
f\left( a\right) }{f\left( b\right) }\right] ^{t}dt\right]   \notag \\
&&  \notag \\
&=&\frac{1}{f\left( b\right) }\frac{f\left( a\right) +f\left( b\right) }{%
\left[ \ln \left( f\left( a\right) -f\left( b\right) \right) \right] ^{2}}+%
\frac{2f\left( a\right) -2f\left( b\right) }{\left[ \ln \left( f\left(
a\right) -f\left( b\right) \right) \right] ^{2}},  \notag
\end{eqnarray}%
and similarly we get,%
\begin{eqnarray}
I_{2} &=&\dint\limits_{0}^{1}t\left( 1-t\right) \left[ \frac{f\left(
b\right) }{f\left( a\right) }\right] ^{t}dt  \label{E15} \\
&&  \notag \\
&=&\frac{1}{f\left( a\right) }\frac{f\left( a\right) +f\left( b\right) }{%
\left[ \ln \left( f\left( b\right) -f\left( a\right) \right) \right] ^{2}}+%
\frac{2f\left( b\right) -2f\left( a\right) }{\left[ \ln \left( f\left(
b\right) -f\left( a\right) \right) \right] ^{2}}.  \notag
\end{eqnarray}%
Putting (\ref{E14}) and (\ref{E15}) in (\ref{E13}), and if we change the
variable $x:=ta+\left( 1-t\right) b$, $t\in \left( 0,1\right) $, we get the
required inequality in (\ref{E9}). This proves the theorem.
\end{proof}

\end{document}